\documentclass[12pt]{article}
\usepackage{amsmath,amssymb,amsthm,latexsym,graphicx}
\newtheorem{theorem}{Theorem}
\newtheorem{lemma}[theorem]{Lemma}

\newtheorem{corollary}[theorem]{Corollary}

\newcommand{\RR}{\mathbb{R}}

\begin{document}

\author{Mark Agranovsky\footnote{Department of Mathematics, Bar-
Ilan University,
Ramat-Gan, 52900, Israel. agranovs@macs.biu.ac.il} and Peter 
Kuchment\footnote{Mathematics Department,
Texas A\& M Univeristy, College Station, TX 77845, USA. 
kuchment@math.tamu.edu}}
\title{Uniqueness of reconstruction and an inversion procedure 
for thermoacoustic and photoacoustic tomography}
\maketitle
\begin{abstract}
The paper contains a simple approach to reconstruction in
Thermoacoustic and Photoacoustic Tomography. The technique works 
for any geometry of point detectors placement and for variable 
sound speed satisfying a non-trapping condition. A uniqueness of 
reconstruction result is also obtained.

\end{abstract}

\section{Introduction}

In the past decade, one witnessed a surge in newly developing
medical imaging  modalities. Their designers pursue the lofty 
goals
of increasing the image resolution, contrast, and safety, while
reducing the costs. A new approach evident in some of these
developments is based upon combining different physical types of
signals in one procedure, with the hope of reducing the deficiencies
of each individual one, and at the same time taking advantage of the
strengths of each. The most well developed example is the {\bf
Thermoacoustic Tomography (TAT)} (also abbreviated as TCT)
\cite{Kruger}. Let us provide a brief description of the TAT
procedure (also see \cite{FPR, Kruger03, Kruger, KuKu, Patch} and
\cite{PAT}-\cite{XWAK}).

A very short radiofrequency (RF) pulse is sent
through a biological object.
\begin{figure}[ht]
\begin{center}
\scalebox{0.7}{\includegraphics{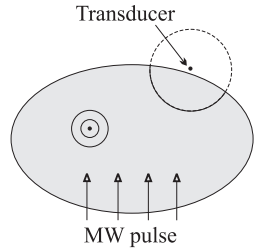}}
\end{center}
\caption{The TAT procedure.}
\label{F:tat}
\end{figure}
\noindent The whole object is assumed to be uniformly irradiated.
Some RF energy is absorbed at each location inside the object. It is
known (e.g., \cite{Kruger,MXW1}) that cancerous cells absorb several
times more energy in the RF range than the healthy ones. Thus, if we
knew the  distribution of the absorption of RF energy function
$f(x)$, this high contrast could provide an easy identification of
cancerous locations. However, RF waves with sufficient penetration
depth are too long to lead to a good resolution, if used for
imaging. Thus, a different mechanism is used for imaging $f(x)$.
Namely, energy absorption causes thermoelastic expansion of the
tissue, which in turn leads to propagation of a pressure wave
$p(x,t)$. This signal is measured by transducers distributed
according to some specifically chosen geometry around the object.
Ultrasound imaging has rather low contrast, but in the TAT
procedure, this is not a problem, since the contrast comes from the
electromagnetic absorption, and ultrasound is only responsible for
high resolution, which it does have. Thus, recovery of $f(x)$ from
the measured characteristics of the pressure $p(x,t)$, is the main
goal of TAT.

The smallness of the ultrasound contrast is in fact a good thing for
TAT. For instance, it has been mostly assumed that the sound speed
is constant. This is the main assumption under which most known
mathematical developments have taken place. In many cases, though,
it is important to take variable sound speed into account, and we
will do so in this paper.

The photoacoustic tomography (PAT) is almost identical to TAT. The
difference is only in how the thermoacoustic signal is generated. In
PAT, a laser pulse is used  instead of an RF one to initiate the
signal \cite{PAT}. The rest of the procedure stays the same. Since
from the mathematical point of view, there is no difference between
TAT and PAT, we will be mentioning TAT only.

Close attention has been paid to developing mathematical methods
required for TAT (almost solely under the constant sound speed
assumption). The mathematical problems of TAT happen to be very
interesting and involving various areas of analysis, in particular
PDEs, integral geometry, microlocal analysis, and spectral theory.
Many of them,  however, also are rather complex (see, e.e. the
survey \cite{KuKu} and references therein)\footnote{We will describe
the mathematical set-up of TAT in the next section.}. After a
substantial effort, major breakthroughs have occurred in the last
couple of years in all relevant issues: uniqueness of
reconstruction, inversion formulas and algorithms, range conditions,
incomplete data problems, and stability. The hardest to come about
were explicit inversion formulas, as well as some uniqueness of
reconstruction results (albeit, for practical geometries, these can
be considered to be resolved \cite{AQ, KuKu}). For quite a while,
only series expansions had been available for the case of
transducers placed along a sphere $S$ surrounding the object 
\cite{Norton1, Norton2}. No formulas were available for non-
spherical surfaces $S$. The first explicit inversion formulas 
were obtained in \cite{FPR} for the spherical geometry in odd 
dimensions and then extended to even dimensions
in \cite{Finch_even}. A different (non-equivalent) set of 
formulas
was developed for any dimension in \cite{Kunyansky} (see also 
such a
formula in $3D$ in \cite{MXW2}). All these formulas are valid for
the spherical geometry only. Another drawback of the inversions 
of \cite{Finch_even, FPR, Kunyansky, MXW2} is that if the source 
$f(x)$ extends beyond the observation surface, the 
reconstruction even of its part that is inside, becomes 
incorrect (see discussion of this phenomenon in \cite{KuKu}). A 
series expansion formula that involves
the spectrum and eigenfunctions of the Dirichlet Laplacian was
obtained in \cite{Kun_series}. Unlike the formulas of
\cite{Finch_even, FPR, Kunyansky, MXW2}, it applies to any 
geometry of measurement and also does not assume that the signal 
comes from the interior of the measurement surface only. It was 
also shown in \cite{Kun_series} that the eigenfunction 
expansions method can be implemented to provide fast and 
accurate reconstructions.

One needs to mention that in the case of the constant sound 
speed,
the reconstruction task can be described as an integral geometry
problem dealing with a spherical mean operator \cite{AKQ, KuKu} 
(see also \cite{Leon_Radon, GGG, Helg_Radon, Helg_groups, 
KuchAMS05, KuchQuinto, Natt_old, Natt_new, Pal_book} and 
references therein for integral geometric methods). This 
relation with
integral geometry all but disappears for non-constant speed.

No explicit inversion formulas are known for non-spherical 
surfaces
or for the case of variable sound speed. There are also some
mysteries surrounding existing formulas. E.g., why do explicit
formulas reconstruct incorrectly the function inside the 
observation
surface, if the function has a part outside the surface? Why do
inversion formulas exist, even so the general machinery of the so
called $\kappa$-operator \cite{GGG1, GGG, Gi}  suggests that one
should not expect such formulas to appear (since one deals with 
a so called ``inadmissible complex'')? Some of these issues are addressed in \cite{KuKu}.

The aim of this article is to develop a very simple general 
operator theory inversion formula, which applies to arbitrary 
geometry of the observation surface (i.e., the surface where 
transducers are placed), variable sound speed under a non-
trapping condition, and functions not necessarily supported 
inside the observation surface. Since the formula involves 
functions of the Dirichlet Laplacian inside $S$, it is not easy 
to apply. We show how it works in some special cases. In 
particular, it leads to eigenfunction 
expansion methods that apply in a wider generality (e.g., in non-
homogeneous media) than those of \cite{Kun_series}. 

One can mention that mathematical problems similar to the ones of
TAT arise in sonar and radar research (e.g., \cite{LQ, NC}).

The structure of the article is a s follows. The next Section
\ref{S:model} sets up the mathematical formulation of the 
problem.
Then, in Section \ref{S:abstract}, we obtain the general 
inversion
formulas. For the clarity sake, they are first derived under the
assumption of a constant sound speed and for a domain of an
arbitrary shape in an odd dimensional space $\RR^n$, so the 
Huygens'
principle holds. Later on, use of the Huygens' principle is 
replaced by local energy decay estimates \cite{Vainb, Vainb2}. 
It is also assumed that functions involved are
smooth, a restriction that can easily be lifted after the final
formulas are derived. Then an abstract result concerning 
hyperbolic
equations in Hilbert spaces is derived that generalizes the one
presented in this simple example. This general formula then is
applied to the case of reconstruction in TAT without assumption 
of constant sound speed. What is relevant, is not the uniformity 
of the background medium, but rather a non-trapping condition. 
Such a condition is absolutely natural, since without it one 
cannot expect any inverse problems of the kind we study to be 
reasonably solvable. We also establish a uniqueness of 
reconstruction result. 
It shows that the data collected is sufficient for recovery of a 
(compactly supported) function, even if its support reaches 
outside the observation surface (albeit the inversion procedure 
recoveres only its interior part). The next Section \ref
{S:remarks} contains additional remarks and discussions. This is 
followed by an Acknowledgements section.

\section{Mathematical model of TAT}\label{S:model}

Let the function $f(x)$ to be reconstructed be compactly supported
in $\RR^n$. Let also $S$ be a closed surface, at each point $y$ of
which one places a point detector that measures the value of the
pressure $p(y,t)$ at any moment $t>0$. It is usually assumed (and this 
is crucial for the validity of the formulas of \cite{Finch_even, FPR, 
Kunyansky, MXW2}) that the object (and thus the support of $f(x)$)
is surrounded by $S$. We will not need this assumption here. 

We assume that the ultrasound speed $c(x)$ is known (e.g., being 
determined by transmission
ultrasound measurements \cite{JinWang}). Then, the pressure wave
$p(x,t)$ satisfies the following problem for the standard wave
equation \cite{Diebold,Tam, MXW1}:
\begin{equation}\label{E:wave}
\begin{cases}
    p_{tt}=c^2(x)\Delta_x p, t\geq 0, x\in\RR^3\\
    p(x,0)=f(x),\\
    p_t(x,0)=0
    \end{cases}
\end{equation}
The task is to recover the initial value $f(x)$ at $t=0$ of the
solution $p(x,t)$ from the measured data, which we will call
$g(y,t)$. Incorporating the data obtained from the measurement, the
set of equations (\ref{E:wave}) extends to become
\begin{equation}\label{E:wave_data}
\begin{cases}
    p_{tt}=c^2(x) \Delta_x p, t\geq 0, x\in\RR^3\\
    p(x,0)=f(x),\\
    p_t(x,0)=0\\
    p(y,t)=g(y,t), y\in S\times\RR^+
    \end{cases}
\end{equation}

\begin{figure}[ht]
\begin{center}
\scalebox{0.7}{\includegraphics{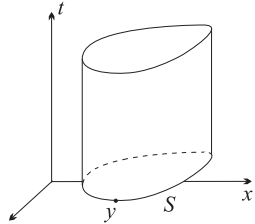}}
\label{F:cylinder}
\end{center}
\caption{An illustration to (\ref{E:wave_data}).}
\label{F:wave}
\end{figure}
The problem now is to find the initial value $f(x)$ in
(\ref{E:wave_data}) from the knowledge of the lateral data $g
(y,t)$ (see Figure \ref{F:cylinder}). Although at a first glance 
one might think that we have insufficient data, this is in fact 
not true. Indeed, there is an additional restriction that the 
solution holds in the whole space, not just inside the cylinder 
$S\times \RR^+$. In
most cases, with this additional information, the data is 
sufficient for recovery of $f(x)$ (see
\cite{KuKu} and references therein for the case of constant 
speed, and further in this paper for the case of a variable 
speed under a non-trapping condition). An additional sometimes 
useful comment (albeit we are not going to use it) is that $p$ 
can be extended as an even function of time, after which it will 
satisfy the wave equation for all values of $t$. The same 
applies to the data $g$ that can also be extended to an even 
function in $t$.

We would like to notice that in the case of constant speed, due 
to the well known Poisson-Kirchhoff formulas \cite[Ch. VI, 
Section 13.2, Formula (15)]{CH} for the solution of (\ref
{E:wave}), there is close
relation (essentially, an equivalence) between this problem and
inversion of the spherical mean operator with centers of spheres 
of integration located on $S$ (see, e.g., \cite{AQ, KuKu} and
references in these papers). This is why many previous
considerations dealt with the spherical mean operator rather than
the wave equation, which allowed for the well developed 
techniques
(or at least hints) of integral geometry to be applied. We cannot
relay upon such techniques, since in our case the wave speed is 
not assumed to be constant.

The standard list of questions one would like to ask in any 
tomographic problem, includes uniqueness of reconstruction, 
inversion formulas and algorithms, their stability, incomplete 
data problems, etc. In this text, we will address only 
uniqueness of reconstruction and inversion. One can find a 
survey of known results on other issues and a bibliography in
\cite{KuKu}.

\section{Discussion of the problem and the main results}\label{S:abstract}

In order not to obstacle consideration, we will assume for
simplicity that the sound speed $c(x)$ and the initial function
$f(x)$ to be recovered are infinite differentiable, an assumption
that can be easily significantly weakened. The closed {\em 
observation surface} $S$ of transducers' locations is assumed to 
be sufficiently ``nice''. E.g.,
assuming that it consists of finitely many transversally
intersecting smooth pieces (e.g., a cube), would suffice both 
for practical applications and for our analytic needs, albeit 
this condition can be weakened much further. Thus, we will not 
dwell on this issue. Notice, that unlike the case of all known 
backprojection type formulas in TAT, we do not assume
that the function to be reconstructed is supported inside the
observation surface $S$, on which the transducers are located.

The truly significant assumptions are the following:
\begin{enumerate}
\item Support of $f(x)$ is compact.

\item The sound speed is strictly positive $c(x)>c>0$ and such 
that $c(x)-1$ has
compact support, i.e. $c(x)=1$ for large $x$.

\item Consider the Hamiltonian system in $\RR^{2n}_{x,\xi}$ with 
the Hamiltonian $H=\frac{c^2(x)}{2}|\xi|^2$:
\begin{equation}
\begin{cases}
x^\prime_t=\frac{\partial H}{\partial \xi}=c^2(x)\xi \\
\xi^\prime_t=-\frac{\partial H}{\partial x}=-\frac 12 \nabla 
\left(c^2(x)\right)|\xi|^2 \\
x|_{t=0}=x_0, \xi|_{t=0}=\xi_0.
\end{cases}
\end{equation}
The solutions of this system are called {\em bicharacteristics} 
and 
their projections into $\RR^n_x$ are {\em rays}.

We will assume that the {\bf non-trapping condition} holds, i.e. 
that all
rays (with $\xi_0\neq 0$) tend to infinity when $t \to \infty$.

\end{enumerate}

The first assumption is very important, albeit it could be 
replaced by a
sufficiently fast decay. Without a sufficiently fast decay 
condition, 
there is no  uniqueness of reconstruction \cite{ABK, AQ, KuKu}. 
The 
second condition assumes that the medium outside 
a bounded domain is homogeneous. The third one prohibits 
trapping. If
trapping does occur, then it is naturally expected that 
solvability of
standard inverse problems and problems like controllability 
should most probably fail.

Let us discuss first some peculiarities of the inverse problem of
recovering the initial value $f(x)$ from the lateral boundary 
data $g
(y,t)$ in (\ref{E:wave_data}). As we have mentioned already, if 
there 
were nothing special about the set-up of this problem (i.e., if 
we were
dealing with boundary value problem for the wave equation in the
cylinder), such an inverse problem would not be uniquely 
solvable.
Indeed, the lateral data alone is clearly insufficient. The 
special
situation is that we are dealing in fact with an {\em 
observability}
problem. Namely, the solution $p(x,t)$ satisfies the wave 
equation in the whole space, not just inside the surface $S$, 
albeit we observe the solution on $S$ only. Since the initial 
function is
compactly supported, and since appropriate non-trapping condition
will be imposed, the energy will be locally decaying (e.g.,
\cite{Lax, Vainb, Vainb2}). Thus, the energy of the solution $p$ 
inside 
$S$ will be decaying. The rate of this decay is different in odd
and even dimensions (\cite{Egorov, Vainb, Vainb2}), exponential 
in the odd and power in the even case. In both cases, the 
solution is summable 
with respect to time as a function in appropriate functional spaces in
the domain $B$ bounded by $S$. We will see that this influences the 
solvability of the problem significantly. 

\subsection{A test case: constant sound speed in an odd dimension}\label{SS:odd}

We start with the simplest case when the dimension $n$ is odd and 
the sound speed is constant and equal to
$1$. We also assume here that the initial function $f(x)$ is smooth 
and compactly supported.  Then both the solution $p(x,t)$ and the 
boundary data 
$g(x,t)$ are smooth and, due to Huygens' principle, compactly
supported in time, when $x\in B$. Here we denote by $B$ the interior 
domain of $S$. Let now $E$ be the operator of harmonic extension of
functions from $S$ to $B$, which, due to standard regularity
theorems, is continuous from $H^s(S)$ to $H^{s+\frac{1}{2}}(B)$ for any
$s>0$. Then we can rewrite the problem (\ref{E:wave_data}) in terms
of the function $u=p-Eg$:
\begin{equation}\label{E:wave_modified}
\begin{cases}
    u_{tt}- \Delta_x u=-E(g_{tt}), t\geq 0, x\in B\\
    u(x,0)=f(x)-(Eg|_{t=0}),\\
    u_t(x,0)=0\\
    u(y,t)=0, y\in S\times\RR^+
    \end{cases}.
\end{equation}
In particular, if the initial function is supported strictly inside 
$B$ (which we do not need to assume),  then $g|_{t=0}=0$, and thus 
$u(x,0)=f(x)$.

Let us now denote by $\Delta_D$ the Dirichlet Laplacian in $B$, 
defined in the standard  way as an unbounded self-adjoint operator 
in $L^2(B)$.
\begin{theorem}\label{T:main_simplified}
The following representation of $u(x,t)$ holds:
\begin{equation}\label{E:general_simplified}
u(x,t)=\int\limits_t^\infty
\left(-\Delta_D\right)^{-\frac12}\sin{((t-\tau)
\left(-\Delta_D\right)^{\frac12})}E(g_{tt})(x,\tau)d\tau.
\end{equation}
In particular,
\begin{equation}\label{E:reconstruction_simplified}
f(x)=(Eg|_{t=0})-\int\limits_0^\infty
\left(-\Delta_D\right)^{-\frac12} \sin{(\tau
\left(-\Delta_D\right)^{\frac12})}E(g_{tt})(x,\tau)d\tau.
\end{equation}
\end{theorem}
\begin{proof}
First of all, the integral in (\ref{E:general_simplified}) makes
sense, since 
$$
\left(-\Delta_D\right)^{-\frac12}\sin{((t-\tau)
\left(-\Delta_D\right)^{\frac12})}
$$ 
is an uniformly bounded strongly
continuous operator function in $L^2(B)$ and
$$
\alpha(t):=E(g_{tt})(\cdot,t)
$$ 
is a smooth compactly supported function
of $t$ with values in $L^2(B)$. Let $\{\lambda_k^2\},\{\psi_k(x)\}$
be the spectrum and an orthonormal basis of eigenfunctions of
$-\Delta_D$, where we assume $\lambda_k>0$. We introduce notations 
$$f_k:=(f,\psi_k),
u_k(t):=(u(\cdot,t),\psi_k), \mbox{ and }\alpha_k(t):=(\alpha(t),
\psi_k)
$$ 
for the Fourier coefficients of $f$, $u$, and $\alpha$, with 
respect to the system $\{\psi_k\}$. Here $(,)$ denotes the scalar 
product in $L^2(B)$.

It is sufficient to prove  (\ref{E:general_simplified}) coupled with
any eigenfunction $\psi_k$:
\begin{equation}\label{E:expanded}
u_k(t)=\int\limits_t^\infty
\left(\left(-\Delta_D\right)^{-\frac12}\sin{((t-\tau)
\left(-\Delta_D\right)^{\frac12})}\alpha(\tau),\psi_k(x)\right)d\tau.
\end{equation}
Using self-adjointness of $\left(-\Delta_D\right)^{-\frac12}\sin{((t-
\tau) \left(-\Delta_D\right)^{\frac12})}$ and the eigenfunction 
property of $\psi_k$, this simplifies to
\begin{equation}\label{E:expanded2}
u_k(t)=\lambda_k^{-1}\int\limits_t^\infty \sin{((t-\tau)
\lambda_k)}\alpha_k(\tau) d\tau.
\end{equation}
It is easy to check that the right hand side in (\ref{E:expanded2})
is the  unique solution vanishing at infinity of the equation
\begin{equation}\label{E:wave_expanded}
u_{tt}=-\lambda_k^2 u- \alpha_k(t).
\end{equation}
However, (\ref{E:wave_expanded}) coincides with the projection of
(\ref{E:wave_modified}) onto the eigenfunction $\psi_k$. Since the
system of eigenfunctions is complete in $L^2(B)$, this proves the
theorem.
\end{proof}
Let us now try to remove the appearance of the extension operator
$E$. In order to do so, we introduce the Fourier coefficients of the
data:
$$
g_k(t)=<g,\frac{\partial \psi_k}{\partial\nu}>:=\int\limits_S g(x,t)
\overline{\frac{\partial \psi_k}{\partial\nu}(x)}dx,
$$
where $\nu$ is the external normal to $S$.
\begin{lemma}\label{E:Fourier_conection}
The following equality holds:
\begin{equation}\label{E:Egk}
(Eg)_k(t)=\lambda_k^{-2}g_k(t).
\end{equation}
In particular,
\begin{equation}\label{E:Egk2}
\alpha_k(t)=\lambda_k^{-2}g^{\prime\prime}_k(t).
\end{equation}
\end{lemma}
\begin{proof}
Indeed, $(Eg)_k=(Eg,\psi_k)=\lambda_k^{-2}(Eg,\Delta\psi_k)$. Using
now Green's  formula and harmonicity of $Eg$, one gets
$(Eg)_k=\lambda_k^{-2}<g,\frac{\partial
\psi_k}{\partial\nu}>=\lambda_k^{-2}g_k(t)$.
\end{proof}
The formulas (\ref{E:expanded2}) and (\ref{E:Egk})-(\ref{E:Egk2}) lead 
to
\begin{equation}\label{E:coef_represent}
f_k=\lambda_k^{-2}g_k(0)-\lambda_k^{-3}\int\limits_0^\infty
\sin{(\lambda_k t)} g_k^{\prime\prime}(t)dt.
\end{equation}
Now integration by parts in (\ref{E:coef_represent}) proves the following
\begin{theorem}\label{T:series_wave}
Function $f(x)$ in (\ref{E:wave_data}) can be reconstructed inside
$B$ from  the data $g$ in (\ref{E:wave_data}), as the following
$L^2(B)$-convergent series:
\begin{equation}\label{E:coef_represent2}
f(x)=\sum\limits_k f_k \psi_k(x),
\end{equation}
where the Fourier coefficients $f_k$ can be recovered using one of the following formulas:
\begin{equation}\label{E:coef_represent2}
\begin{cases}
f_k=\lambda_k^{-2}g_k(0)-\lambda_k^{-3}\int\limits_0^\infty
\sin{(\lambda_k t)} g_k^{\prime\prime}(t)dt,\\ 
f_k=\lambda_k^{-2}g_k(0)+\lambda_k^{-2}\int\limits_0^\infty
\cos{(\lambda_k t)} g_k^{\prime}(t)dt, \mbox{ or } \\
f_k=-\lambda_k^{-1}\int\limits_0^\infty \sin{(\lambda_k t)}g_k(t)dt
=-\lambda_k^{-1}\int\limits_0^\infty \int\limits_S\sin{(\lambda_k t)}g(x,t)\overline{\frac{\partial \psi_k}{\partial\nu}(x)}dxdt,
\end{cases}
\end{equation}
where
$$
g_k(t)=\int\limits_{S}g(x,t)\overline{\frac{\partial \psi_k}{\partial
\nu}(x)}dx
$$
and $\nu$ denotes the external normal to $S$.
\end{theorem}
\begin{corollary}\label{C:better_smoothness}
The statement of the Theorem with the third of the coefficient formulas 
in (\ref{E:coef_represent2}) holds under much milder conditions on the 
function $f(x)$. For instance, it is sufficient to assume that $f\in 
H^s_{comp}(\RR^n)$ for some $s>\frac 12$.
\end{corollary}
Indeed, one can approximate such a function $f$ in the space $H^s_{comp}
(\RR^n)$ by functions from $C_0^\infty(\RR^n)$, thus getting 
convergence of the data $g(x,t)$ in $L^2_{comp}(S\times\RR)$. Then the 
third formula in (\ref{E:coef_represent2}) extends by continuity.

\subsection{An abstract theorem}

It is not hard to establish an abstract analog of Theorems \ref
{T:main_simplified} and \ref{T:series_wave}. 
\begin{theorem}\label{T:abstract}
Let $H$ be a separable Hilbert space and $A$ be a positive self-adjoint 
operator in $H$ with discrete spectrum $\{\lambda_k^2\}$ and 
orthonormal basis of eigenvectors $\psi_k$. Suppose that function $u:
\RR^+ \to H$ solves the equation
\begin{equation}\label{E:wave_abstract}
u_{tt}=-Au-\alpha(t)
\end{equation} 
and that the following conditions are satisfied:
\begin{enumerate}
\item The function $u$ and its derivative tend to zero when 
$t \to\infty$:
$$
\lim\limits_{t\to\infty} \|u(t)\|_H=\lim\limits_{t\to\infty} \|u_t(t)\|
_H=0.$$
\item The function $\alpha(t)$ belongs to $L^1(\RR^+,H)$.

\end{enumerate}
Then
\begin{enumerate}
\item The following representation of $u(t)$ holds:
\begin{equation}\label{E:abstract_simplified}
u(t)=\int\limits_t^\infty
A^{-\frac12}\sin{((t-\tau)
A^{\frac12})}\alpha(\tau)d\tau.
\end{equation}
In particular,
\begin{equation}\label{E:reconstruction_abstract}
u(0)=-\int\limits_0^\infty
A^{-\frac12} \sin{(\tau
A^{\frac12})}\alpha(\tau)d\tau.
\end{equation}
\item The coefficients $u_k$ of the Fourier series expansion 
$$
u(0)=\sum\limits_k u_k\psi_k
$$
can be found as follows:
$$
u_k=-\lambda_k^{-1}\int\limits_0^\infty
\sin{(\tau
\lambda_k)}\alpha_k(\tau)d\tau,
$$
where $\alpha_k(t)=(\alpha(t),\psi_k)_H$.
\end{enumerate}
\end{theorem}
\begin{proof} The proof repeats the proofs of Theorems \ref
{T:main_simplified} and \ref{T:series_wave}.\end{proof}

\subsection{An application to TAT in inhomogeneous media}

Let us consider now the TAT problem (\ref
{E:wave_data}) in the presence 
of an inhomogeneous background, assuming the 
conditions imposed on the 
sound speed $c(x)$ in the beginning of Section \ref{S:abstract}.
Using the same harmonic extension trick as in the Section \ref{SS:odd}, 
we get
\begin{equation}\label{E:wave_variable}
\begin{cases}
    u_{tt}- c^2(x)\Delta_x u=-E(g_{tt}), t\geq 0, x\in B\\
    u(x,0)=f(x)-(Eg|_{t=0}),\\
    u_t(x,0)=0\\
    u(y,t)=0, y\in S\times\RR^+
    \end{cases}.
\end{equation}
Consider the Hilbert space $H=L^2(B,c^{-2}(x)dx)$, i.e., the weighted 
$L^2$ space with the weight $c^{-2}(x)$. On 
this space, the naturally defined operator 
$$
A=-c^2(x)\Delta
$$ 
in $B$ with zero Dirichlet conditions on $S$ is self-adjoint, 
positive, and has discrete spectrum $\{\lambda_k^2\} (\lambda_k>0)
$ with eigenfunctions $\psi_k(x)
\in H$. Now, since we are in fact dealing with the unobstacled 
whole space wave propagation (the surface $S$ is not truly a 
boundary, but just an 
observation surface) and since we assumed the sound speed constant at 
infinity and non-trapping, the local energy decay type estimates of 
\cite{Vainb, Vainb2} (see also \cite[Theorem 2.104]{Egorov}) apply. 
They, in combination with the standard Sobolev trace theorems, imply in 
particular that the conditions of Theorem \ref{T:abstract} both on $u
(t):=u(\cdot,t)\in H$ and $\alpha(t)=E(g_{tt})$ are satisfied. We thus 
get the following

\begin{theorem}\label{T:main_speed}
\begin{enumerate}
\item The function $f(x)$ in (\ref{E:wave_variable}), and thus in (\ref
{E:wave_data}) can be reconstructed inside $B$ as follows:
\begin{equation}\label{E:reconstruction_variable}
f(x)=(Eg|_{t=0})-\int\limits_0^\infty
A^{-\frac12} \sin{(\tau
A^{\frac12})}E(g_{tt})(x,\tau)d\tau.
\end{equation}

\item Function $f(x)$ can be reconstructed inside
$B$ from the data $g$ in (\ref{E:wave_data}), as the following
$L^2(B)$-convergent series:
\begin{equation}\label{E:coef_variable}
f(x)=\sum\limits_k f_k \psi_k(x),
\end{equation}
where the Fourier coefficients $f_k$ can be recovered using one of 
the following formulas:
\begin{equation}\label{E:coef_variable2}
\begin{cases}
f_k=\lambda_k^{-2}g_k(0)-\lambda_k^{-3}\int\limits_0^\infty
\sin{(\lambda_k t)} g_k^{\prime\prime}(t)dt,\\ 
f_k=\lambda_k^{-2}g_k(0)+\lambda_k^{-2}\int\limits_0^\infty
\cos{(\lambda_k t)} g_k^{\prime}(t)dt, \mbox{ or } \\
f_k=-\lambda_k^{-1}\int\limits_0^\infty \sin{(\lambda_k t)}g_k(t)dt
=-\lambda_k^{-1}\int\limits_0^\infty \int\limits_S\sin{(\lambda_k 
t)}g (x,t)\overline{\frac{\partial \psi_k}{\partial\nu}(x)}dxdt,
\end{cases}
\end{equation}
where
$$
g_k(t)=\int\limits_{S}g(x,t)\overline{\frac{\partial \psi_k}
{\partial \nu}(x)}dx
$$
and $\nu$ denotes the external normal to $S$.
\end{enumerate}
\end{theorem}
\begin{proof}The only thing that might have been different is 
deriving the Fourier coefficient representations (\ref
{E:coef_variable2}) by 
coupling the expression $E(g_{tt})$ with $\psi_k$ in (\ref
{E:reconstruction_variable})
in the weighted $L^2$ space. This, however, causes no problem, 
since
$$
\begin{array}{c}
\int\limits_B E(g_{tt})(x,t) \overline{\psi_k(x)}c^{-2}(x)dx\\
= \lambda_k^{-2}\int\limits_B E(g_{tt})(x,t) \overline{c^2(x)\Delta
\psi_k(x)}c^{-2}(x)dx\\
=\lambda_k^{-2}\int\limits_B E(g_{tt})(x,t) \overline{\Delta\psi_k
(x)}dx,
\end{array}
$$
from where the proof proceeds as before.
\end{proof}

Uniqueness of reconstruction of $f(x)$ inside $B$ obviously 
follows from this Theorem. It, however, requires some additional 
arguments, if one wants to derive uniqueness of recovery of $f(x)$ 
outside $S$. We thus sketch an independent proof of uniqueness of 
reconstruction of the whole function $f(x)$, which does not rely 
upon the previous theorem.
\begin{theorem}\label{T:uniqueness}
Under the non-trapping conditions formulated above, compactly 
supported function $f(x)$ is uniquelly determined by the data $g$. 
(No assumption of $f$ being supported inside $S$ is imposed.)
\end{theorem}
\begin{proof}Suppose that the data $g$ 
vanishes. Then, inside $B$, due to local energy decay \cite
{Egorov, Vainb, Vainb2}, $u(t)$, as an $H$-valued function, is 
integrable over $\RR$. Taking Fourier transform in time, we get a 
continuous $H$-valued function 
$\hat{u}(\lambda)$ that satisfies the equation $(-\lambda^2+A)\hat
{u}(\lambda)=0$ almost everywhere. However, the spectrum of $A$ is 
discrete, and thus $\hat{u}(\lambda)$ must vanish for all but 
discrete 
set of values of $\lambda$. Due to continuity of $\hat{u}(\lambda)
$, this implies that $\hat{u}(\lambda)$, and thus $u(t)$ vanish. 
This gives vanishing of $f$ inside $B$. (This could also be 
derived immediately from the preceding theorem). Now, the function 
$u$ extends to a solution of the wave equation outside $B$ with 
matching Cauchy data on $S$. After Fourier transform in time, we 
get for each value of $\lambda$ a solution $\hat{u}(x,\lambda)$ of 
the equation $c^2(x)
\Delta \hat{u}+\lambda^2\hat{u}=0$ in $\RR^n$, vanishing inside $S
$. Then standard uniqueness of continuation theorems (remember that 
$c$ is smooth and non-vanishing) prove that $\hat{u}$, and hence $u
$, is zero.\end{proof}

\section{Further discussion and remarks}\label{S:remarks}
\begin{itemize}

\item The results of this paper are in some sense the realizations of the remarks made in \cite{FPR} concerning inversion by solving that wave equation in reverse time direction and in \cite{AKQ} concerning inversions by eigenfunction expansions (improved and implemented in \cite{Kun_series}).

\item It is instructive to check how the last of the formulas (\ref
{E:coef_represent2}) works in $1D$ in the case of constant speed. 
Let 
us assume that $B=[0,1]$ and $f(x)$ is supported in $B$. Then $
\lambda_k=k\pi$ and $\psi_k(x)=\sqrt{2}\sin{k\pi x}$. The 
d'Alambert 
formula claims that $u(x,t)=\frac{1}{2}\left(f(x-t)+f(x+t)\right)$ 
and thus $g
(0,t)=
\frac{1}{2}f(t)$, $g(1,t)=\frac{1}{2}f(1-t)$ for $t\in [0,1]$ and 
$g=0$ 
for $t>1$. The normal derivatives of the eigenfunctions are
$$
\frac{\partial \psi_k}{\partial \nu}(0)=-\sqrt{2}k\pi, \frac
{\partial 
\psi_k}{\partial \nu}(1)=\sqrt{2}(-1)^k k\pi. 
$$
Now substitution of these expression into the last of the formulas 
(\ref
{E:coef_represent2}) after a short computation amounts to the 
standard 
formula for the Fourier coefficients of $f$ with respect to the 
system $
\psi_k$.

\item Another instructive computation is to compare with the 
expansion formulas from \cite{Kun_series} that were obtained under 
the assumption of a constant sound speed. Consider the $3D$ case. 
One needs to notice that the 
data used in \cite{Kun_series}, although denoted by $g(t)$ as ours, 
is different. In \cite{Kun_series}, it was, up to a factor $4\pi$, 
the integral over the sphere of radius $t$, while ours corresponds 
to the solution of the wave equation (\ref{E:wave_data}). Taking 
this into account, our data (call it temporarily $\tilde{g}$), is 
equal to $\dfrac
{d}{dt}\left(\frac{g(t)}{4\pi t}\right)$. Substitution of this $
\tilde
{g}$ instead of $g$ into the last formula of (\ref
{E:coef_represent2}) and one integration by parts immediately lead 
to the formulas (10), (12), and (13) in \cite{Kun_series}. Note 
however 
that, unlike in \cite{Kun_series}, no knowledge of the whole space 
Green's function is needed. In fact, our formulas hold in the case 
of variable non-trapping sound speed, when one cannot write this 
Green's 
function explicitly and where the method of \cite{Kun_series} fails.

\item Notice that, like in \cite{Kun_series}, and unlike \cite
{Finch_even, FPR, Kunyansky, MXW2}, the inversion procedures 
derived 
in this paper do not require the function $f(x)$ to be supported 
inside the observation surface $S$. Even if the (compact) support 
of $f
$ reaches outside $S$, reconstruction inside $S$ is correct. It is 
known, however, that in this case the backprojection formulas from 
\cite{Finch_even, FPR, Kunyansky, MXW2} produce incorrect 
reconstructions inside $S$. The reason is that the formulas from 
\cite
{Finch_even, FPR, Kunyansky, MXW2} have the information built in 
about 
the support of $f$ being inside. In particular, they use, in our 
notations, only the values of $t$ up to the diameter of the domain 
$B$. It is clear that in this case reconstruction of functions 
supported 
not necessarily inside $S$ is not unique, since, due to the finite 
speed of propagation, information from all support of $f$ might not 
reach $S$ during this time. On the other hand, formulas of this 
paper, as well as of \cite{Kun_series} use the information for all 
values of time $t$. 

\item It has been assumed throughout the paper that the initial 
velocity $p_t(x,0)$ in (\ref{E:wave_data}) is equal to zero. It is 
clear, however, that this is of no importance for our 
considerations, and one can reconstruct the (possibly, non-zero) 
initial velocity as 
well, using formulas (\ref{E:general_simplified}) and (\ref
{E:abstract_simplified}). These formulas were derived without using 
the assumption that $p_t(x,0)=0$.

\item We went along with accepted models \cite{Diebold, JinWang, 
Tam} of non-homegenous media in TAT that lead to (\ref
{E:wave_data}). In these models, variations of density are 
neglected. However, the abstract Hilbert space formalism that we 
provided, allows one to include also variations of density without 
much of a change in formulas and methods. This just leads to a more 
general operator $A$ in $L^2(B)$.

\item The integral formulas (\ref{E:general_simplified}), (\ref
{E:abstract_simplified}), and (\ref{E:reconstruction_variable}) use 
functions of the Laplace (or a more general) operator in $L^2(B)$ 
with Dirichlet boundary conditions on $S$. It is clear how to use 
such formulas in combinations with eigenfunction expansions of the 
operator, as we did in this text. It is, however, less clear how 
such formulas could lead to analytic backprojection type 
inversions, such as the ones in \cite{Finch_even, FPR, 
Kunyansky, 
MXW2}. A Fourier transform consideration that leads to Helmholtz 
type equations suggests that there might be a link to \cite
{Kunyansky}, and 
thus one might hope to produce inversion formulas of \cite
{Kunyansky} obtained for the case of $S$ being a sphere, as a 
consequence of our formulas. The following simple comments 
might be useful for the further progress in this direction.

First of all, the gist of the consideration of Section \ref
{SS:odd} 
is the following simple relation (derived readily using Green's 
formula):
$$
(\Delta_D u)_k= (\Delta u)_k +\int\limits_S u\frac{\partial 
\psi_k}
{\partial \nu} dA(x).
$$

\indent Secondly, reconstruction formulas (\ref
{E:coef_variable})-
(\ref
{E:coef_variable2}) can be combined into the following integral formula:
\begin{equation}\label{E:integral_formula}
f(x)=-\int\limits_0^\infty\int\limits_S \partial_{\nu_y}K(x,y,t)g
(y,t)dA(y)dt,
\end{equation}
where the kernel $K$ is
\begin{equation}\label{E:kernel}
K(x,y,t)=\sum\limits_k\lambda_k^{-1} \sin (\lambda_kt)\psi_k(x)
\psi_k(y).
\end{equation}

It is readily checked that  $K$ satisfies the following 
initial-boundary value problem:
\begin{equation}\label{E:K}
\begin{cases}
    K_{tt}- c^2(y)\Delta_y u=0, t\geq 0, y\in B\\
    K(x, y, 0)= 0,\\
    K_t(x, y, 0)= \delta(y-x), \\
    K(x,y,t)=0, y\in S\times\RR^+
    \end{cases}.
\end{equation}

In this context, the inversion formula (21) can be formally 
derived from the Green's formula:
$$
\begin{array}{c}
0= \int_{B \times [0,T]} [(K_{tt}- c^2(y)\Delta_y K)u- 
K (u_{tt}-c^2(y)\Delta u)]c^{-2}(y)dydt\\
=\int_{S \times [0,T]} [K \partial_{\nu_y}g-\partial_{\nu_y}  K 
g] dA(y)dt  
+\int_{B \times (\{0\} \cup \{T\})} [K_t u- K u_t] dy,
\end{array}
$$
by substituting the initial-boundary conditions, letting
$T \to \infty$ and using the
vanishing of the solutions when $T \to \infty.$

It might be possible in the case of $S$ being a sphere 
to convert the inversion formula for $f$ into a backprojection 
type one.  

\item For the constant sound speed case, it has been known for 
quite some time \cite{ABK, AQ, Ku93} that any closed surface $S$ 
provides uniqueness of reconstruction. The method of \cite{AQ} 
is not applicable, since it relies upon constant sound speed. 
However, spectral methods of \cite{ABK, Ku93} are applicable, as 
it is shown in the proof of Theorem \ref{T:uniqueness}. 

Theorem \ref{T:uniqueness} guarantees unique recovery of any 
compactly supported function $f(x)$, even if its support is not 
confined to $B$. However, microlocal arguments show that 
reconstruction should become unstable for some parts of $f$ 
outside the observation surface $S$. This is related to the 
existence in this case of the so called ``visible'' or 
``audible'' zones of the 
wave front set of $f(x)$, as well as those that are not 
``visible'' (``audible'') \cite{LQ, Pal_book, Quinto}. See \cite
{AKQ, AmbKuch_range, KuKu} for a brief discussion of this issue.

\item Albeit this might not be clear from the text of this 
paper, what has led the authors to this work, was an approach 
through the range conditions of the spherical mean operator 
described in \cite{AKQ}. Indeed, the line of thought was that 
the whole possibility of inversion is based upon a very special 
type of the boundary data $g$ involved. As we have already 
explained, using arbitrary functions as the data $g$ would lead 
to impossibility of reconstruction. 
Thus, the basis of our 
approach was to use our knowledge of the special features of the 
data $g$ to derive inversions. It can be shown (albeit we do not do 
this in this text) that decay with time of solutions inside $B$ is 
directly linked (in the case of constant sound speed) to the range 
descriptions of \cite{AKQ}. Continuing this consideration, one can 
also obtain some analogs of range descriptions for the case of 
variable sound speed.
\end{itemize}

\section*{Acknowledgments}

The work of the second author was partially supported by the NSF DMS
grant  0604778. This work was done when both authors visited the
Isaac Newton Institute of Mathematical Sciences in Cambridge. The
authors express their gratitude to the NSF and the INI for this
support. The authors thank L.~Kunyansky, V.~Palamodov, B.~Vainberg, 
and E.~Zuazua for information and useful discussions.


\begin{thebibliography}{99}
\bibitem{ABK}  M. Agranovsky, C. Berenstein, and P. Kuchment,
Approximation by spherical waves in $L^p$-spaces, J. Geom. Anal. {\bf 6}
(1996), no. 3, 365--383.

\bibitem{AKQ} M.~Agranovsky, P.~Kuchment, and E.~T.~Quinto, Range
descriptions for the spherical mean Radon transform,
J. Funct. Anal {\bf 248} (2007), 344--386.

\bibitem{AQ}  M. Agranovsky and E.T. Quinto, Injectivity sets for 
the Radon transform over circles and complete systems of radial
functions, J. Funct. Anal. {\bf 139} (1996), 383-414.

\bibitem{AmbKuch_range} G.~Ambartsoumian and P.~Kuchment, A range 
description for the planar circular Radon transform, SIAM 
J. Math. Anal. {\bf 38(2)} (2006), 681--692.

\bibitem{AmbPatch} G. Ambartsoumian and S.~Patch, Thermoacoustic 
tomography - implementation of exact backprojection formulas,
arXiv:math.NA/0510638.

\bibitem{CH} R.~Courant and D.~Hilbert, \textit{Methods of
Mathematical Physics, Volume II Partial Differential Equations},
Interscience, New York, 1962.

\bibitem{Diebold} G.~J.~Diebold, T.~Sun, M.~I.~Khan, Photoacoustic
monopole radiation in one, two, and three dimensions, Phys. Rev. 
Lett. {\bf 67} (1991), no. 24, 3384--3387.

\bibitem{Egorov} Yu.~V.~Egorov and M.~A.~Shubin, \textit{Linear
Partial Differential Equations. Foundations of the Classical
Theory},  in \textit{Partial Differential Equations. I.},
Yu.~V.~Egorov and M.~A.~Shubin (Editors), Encyclopaedia of
Mathematical Sciences, v. 30, Springer Verlag 1992, pp. 1--259.

\bibitem{Leon_Radon}L.~Ehrenpreis, \textit{The Universality of the 
Radon Transform}, Oxford Univ. Press 2003.

\bibitem{Finch_even} D.~Finch, M.~Haltmeier, and Rakesh, Inversion of
spherical means and the wave equation in even dimensions, preprint
arXiv math.AP/0701426.

\bibitem{FPR} D.~Finch, S.~Patch, and Rakesh, Determining a function from
its mean values over a family of spheres, SIAM J. Math. Anal. {\bf
35}  (2004), no. 5, 1213--1240.

\bibitem{GGG1} I.~Gelfand, S.~Gindikin, and M.~Graev, Integral geometry
in affine and projective spaces, J. Sov. Math. 18(1980), 39--167.

\bibitem{GGG} I.~Gelfand, S.~Gindikin, and M.~Graev,
\textit{Selected Topics in Integral Geometry}, Transl. Math. Monogr.
v. 220, Amer. Math. Soc., Providence RI, 2003.

\bibitem{Gi} S. Gindikin, Integral geometry on real quadrics, in
\textit{ Lie groups and Lie algebras: E. B. Dynkin's Seminar},
23--31, Amer. Math.  Soc. Transl. Ser. 2, 169, Amer. Math. Soc.,
Providence, RI, 1995.

\bibitem{Helg_Radon} S. Helgason, \textit{The Radon Transform},
Birkh\"{a}user, Basel 1980.

\bibitem{Helg_groups} S.~Helgason, \textit{Groups and Geometric
Analysis}, Amer. Math. Soc., Providence, R.I. 2000.

\bibitem{JinWang}X.~Jin, and L.~V.~Wang, Thermoacoustic tomography
with correction for acoustic speed variations, Physics in Medicine
and  Biology {\bf 51} (2006), 6437?-6448.

\bibitem{Kruger03}R.~A.~Kruger, W.~L.~Kiser, D.~R.~Reinecke, and G.~A.~Kruger,
Thermoacoustic computed tomography using a conventional linear
transducer array,  Med. Phys. {\bf 30 (5)} (2003), 856--860.

\bibitem{Kruger} R.~A. Kruger, P.~Liu, Y.~R.~Fang, and C.~R.~Appledorn,
Photoacoustic ultrasound (PAUS)reconstruction tomography, Med.
Phys. {\bf 22} (1995), 1605-1609.

\bibitem{Ku93} P.~Kuchment, unpublished, 1993.

\bibitem{KuchAMS05} P.~Kuchment, Generalized Transforms of Radon Type 
and Their Applications, in \cite{OlafQuinto}, pp. 67--91.

\bibitem{KuKu} P.~Kuchment and L.~Kunyansky, Mathematics of 
thermoacoustic and photoacoustic tomography, preprint arXiv:0704.0286v1 
[math.AP], submitted.

\bibitem{KuchQuinto} P.~Kuchment and E.~T.~Quinto, Some
problems of integral geometry arising in tomography,
chapter XI in \cite{Leon_Radon}.

\bibitem{Kunyansky} L.~Kunyansky,
Explicit inversion formulae for the spherical mean Radon transform, 
Inverse Problems \textbf{23} (2007), pp. 373--383.

\bibitem{Kun_series}L.~Kunyansky, A series solution and a fast 
algorithm for the inversion of the spherical mean Radon transform, 
preprint arXiv math.AP/0701236,
submitted.

\bibitem{Lax} P.~D.~Lax, R.~S.~Phillips, Scattering theory, 2nd 
edition, Pure and Applied Mathematics, 26. Academic Press, Inc., 
Boston, MA, 1989.

\bibitem{LQ} A.~K.~Louis and E.~T.~ Quinto, Local
tomographic methods in Sonar, in \textit{Surveys on
solution methods for inverse problems}, pp. 147-154,
Springer, Vienna, 2000.

\bibitem{Natt_old} F.~Natterer, \textit{The mathematics of computerized 
tomography}, Wiley, New York, 1986.

\bibitem{Natt_new} F.~Natterer and F.~W\"{u}bbeling, \textit
{Mathematical
Methods in Image Reconstruction}, Monographs on Mathematical
Modeling and Computation v. 5, SIAM, Philadelphia, PA 2001.

\bibitem{NC}C.~J.~Nolan and M.~Cheney, Synthetic aperture
inversion, Inverse Problems {\bf 18}(2002), 221--235.

\bibitem{Norton1} S.~J.~Norton, Reconstruction of a two-dimensional 
reflecting
medium over a circular domain: exact solution, J. Acoust. Soc. Am. {\bf 
67} (1980), 1266-1273.

\bibitem{Norton2} S.~J.~Norton and M.~Linzer, Ultrasonic reflectivity 
imaging in three dimensions: exact inverse scattering solutions for 
plane, cylindrical,
and spherical apertures, IEEE Transactions on Biomedical Engineering, 
{\bf 28} (1981), 200--202.

\bibitem{OlafQuinto} G.~Olafsson and E.~T.~Quinto (Editors), \textit
{The Radon Transform, Inverse Problems, and Tomography. American 
Mathematical Society Short Course January 3--4, 2005, Atlanta, 
Georgia}, Proc. Symp. Appl. Math., v. 63, AMS, RI 2006.

\bibitem{Pal_book} V.~P.~Palamodov, \textit{Reconstructive Integral 
Geometry}, Birkh\"{a}user, Basel 2004.

\bibitem{Patch} S.~K.~Patch, Thermoacoustic tomography - consistency 
conditions and the partial scan problem, Phys. Med. Biol. {\bf 49} 
(2004), 1--11.

\bibitem{Quinto} E.~T.~Quinto, Singularities of the X-ray transform
and limited data tomography in ${\mathbf R}^2$ and $\mathbf R^3$,
{SIAM J.  Math. Anal.} {24}(1993), 1215--1225.

\bibitem{Tam}A.~C.~Tam, Applications of photoacoustic sensing 
techniques, Rev. Mod.
Phys. {\bf 58} (1986), no. 2, 381--431.

\bibitem{Vainb} B.~Vainberg, The short-wave asymptotic behavior of the 
solutions of
stationary problems, and the
asymptotic behavior as $t\to\infty$ of the solutions of nonstationary 
problems, Uspehi
Mat. Nauk {\bf 30} (1975), no. 2(182), 3--55 (Russian). English 
translation: Russian Math.
Surveys {\bf 30}  (1975), no. 2, 1--58.

\bibitem{Vainb2} B.~Vainberg, \textit{Asymptotics methods in the 
Equations of Mathematical Physics}, Gordon and Breach 1989. 
(Translation of the Russian 1982 edition).

\bibitem{PAT} X.~Wang, Y.~Pang, G.~Ku, X.~Xie, G.~Stoica, L.~Wang, 
Noninvasive laser-induced
photoacoustic tomography for structural and functional {\em in vivo} 
imaging of the brain,
Nature Biotechnology, {\bf 21} (2003), no. 7, 803--806.

\bibitem{MXW1} M.~Xu and L.-H.~V.~Wang, Time-domain reconstruction for
thermoacoustic tomography in a spherical geometry, IEEE Trans.
Med. Imag. {\bf 21} (2002), 814-822.

\bibitem{MXW2}M.~Xu and L.-H.~V.~Wang, Universal back-projection algorithm for photoacoustic
computed tomography, Phys. Rev. E {\bf 71} (2005), 016706.

\bibitem{XFW}Y.~Xu, D.~Feng, and L.-H.~V.~Wang, Exact frequency-domain
reconstruction for thermoacoustic tomography: I. Planar geometry,
IEEE Trans. Med. Imag. {\bf 21} (2002), 823-828.

\bibitem{XXW}Y.~Xu, M.~Xu, and L.-H.~ V.~Wang, Exact frequency-domain
reconstruction for thermoacoustic tomography: II. Cylindrical
geometry, IEEE Trans. Med. Imag. {\bf 21} (2002), 829-833.

\bibitem{XWAK} Y. Xu, L. Wang, G. Ambartsoumian, and P. Kuchment,
Reconstructions in limited view thermoacoustic tomography, Medical Physics 31(4)
April 2004, 724-733.

\end{thebibliography}
\end{document}